\documentclass[12pt]{amsart}

\usepackage{graphicx,indentfirst}
\usepackage{amsmath,amssymb,mathrsfs}
\usepackage{amsthm,amscd}
\usepackage{verbatim}
\usepackage{appendix}
\usepackage{enumitem,titletoc}
\usepackage{appendix}
\usepackage{imakeidx}
\makeindex[columns=2, title=Alphabetical Index]
\usepackage{fancyhdr}
\usepackage{amsfonts,color}
\usepackage[all]{xy}
\usepackage{tikz-cd}
\usepackage{syntonly}
\usepackage{fancyhdr}
\newtheorem{theorem}{Theorem}[section]
\newtheorem{lemma}{Lemma}[section]

\usepackage[left=2.55cm, right=2.55cm, bottom=3cm, top=3cm]{geometry}

\usepackage{tikz}
\usepackage{extarrows}
\usepackage{hyperref}

\def\XXint#1#2#3{{\setbox0=\hbox{$#1{#2#3}{\int}$ }
\vcenter{\hbox{$#2#3$ }}\kern-.6\wd0}}

\newtheorem{prop}{Proposition}[section]

\newtheorem{corollary}{Corollary}[section]

\allowdisplaybreaks

\newcommand{\ddbar}{i\partial\bar\partial}
\newcommand{\ric}{\mathrm{Ric}}
\newcommand{\innpro}[1]{\langle#1\rangle}
\newcommand{\bk}[1]{\Big(#1\Big)}

\newcommand{\xk}[1]{\big(#1\big)}

\makeindex

\DeclareMathOperator{\vol}{Vol}

\DeclareMathOperator{\tr}{tr}

\numberwithin{equation}{section}
\allowdisplaybreaks

\hypersetup{
    colorlinks=true,
    citecolor=green,
    filecolor=green,
    linkcolor=blue,
    urlcolor=black
}

\begin{document}

\address{Department of Mathematics \& Computer Science, Rutgers University, Newark, NJ 07102}

\email{bguo@rutgers.edu}

\address{Department of Mathematics, Rutgers University, Piscataway, NJ 08854}
\email{jiansong@math.rutgers.edu}

\title{ Local noncollapsing for complex Monge-Amp\`ere equations}
\author{Bin Guo \and Jian Song}\date{}

\begin{abstract} We prove a local volume noncollapsing estimate for K\"ahler metrics induced from a family of complex Monge-Amp\`ere equations, assuming a local Ricci curvature lower bound. This local volume estimate can be applied to establish various diameter  and gradient estimate.

\end{abstract}

\maketitle

\section{Introduction}

The non-local collapsing of Ricci flow on Riemannian manifolds is a fundamental theorem for the compactness of Ricci flows \cite{P}, which states that along the flow the volume of the geodesic balls is uniformly bounded below by a positive constant, if the scalar curvature along the flow is bounded. In K\"ahler geometry, it has generated great interest to study  degenerating families of K\"ahler metrics, satisfying certain complex Monge-Amp\`ere (MA) equations \cite{S}. Analogous to Ricci flow, it is tempting to study the non-local collapsing of K\"ahler metrics along a degenerating family. 

\medskip

Complex Monge-Amp\`ere equations have been an important tool in the construction of canonical K\"ahler metrics, ever since Yau's solution to the Calabi conjecture \cite{Y}. {\em A priori} estimates on the K\"ahler potentials are the keys to study the geometry of the K\"ahler metrics. Using pluripotential theory, Ko\l{}odziej \cite{K} first proved the sharp $C^0$ estimate for complex 
MA equations, assuming the right-side belongs to some Orlicz space. Ko\l{}odziej's approach was generalized to a family of complex MA equations, allowing the K\"ahler classes to degenerate in \cite{EGZ, DP, DZ, Z}. Recently, a PDE-based method \cite{GPT} was utilized  to give a new and uniform proof of $C^0$ estimates for a class of fully nonlinear partial differential equations, which include in particular the complex MA and Hessian equations. In \cite{FGS}, X. Fu and the authors first revealed the relationship between the $C^0$ estimate of the K\"ahler potentials with the global diameter bound of the K\"ahler metrics, generalizing the distance estimate in \cite{S}. The current paper aims to study how the $C^0$ estimates of K\"ahler potentials affect the geometry locally.

\medskip

Let $(X,\omega_X)$ be a compact K\"ahler manifold of complex dimension $n$. Suppose $\chi$ is a smooth closed $(1,1)$-form such that the class $[\chi]$ is {\em nef}, which means the class $[\chi]$ lies in the closure of the K\"ahler cone of $X$. For $t\in (0,1]$, we consider a family of smooth closed forms
$$\hat \omega_t = \chi + t \,\omega_X.$$ Though $\hat \omega_t$ may not be K\"ahler forms, each $[\hat \omega_t]$ is a K\"ahler class. We consider the following complex Monge-Amp\`ere equation
\begin{equation}\label{eqn:MA}
(\hat \omega_t + \ddbar \varphi_t)^n = c_t e^F \omega_X^n,\quad \sup_X \varphi_t = 0,
\end{equation}
where $F$ is a smooth function normalized such that $$\int_X e^F \omega_X^n = \int_X \omega_X^n = V, ~c_t = \frac{V_t}{V} = O(t^{n - \nu}),~V_t = \int_X \hat \omega_t^n$$ 
and $\nu\in \{0,\ldots, n\}$ is the numerical dimension of the class $[\chi]$.

We denote $\omega_t = \hat\omega_t + \ddbar\varphi_t$ to be the K\"ahler metric satisfying \eqref{eqn:MA} and write $g_t$ for the associated Riemannian metric. The following is our main theorem.
\begin{theorem}\label{thm:main}
For any $p>n$ and $R_0\in (0,1]$, if $\omega_t$ solves (\ref{eqn:MA}) for $t\in (0,1]$ and  the Ricci curvature $\ric(g_{t_0})$ satisfies  
$$\ric(g_{t_0})\ge - \frac{K^2}{R_0^{2}},  ~on~ B_{g_{t_0}}(z_0, 2 R_0) $$
 for some $K \ge 0$, $t_0\in (0, 1]$ and $z_0\in X$,  then
\begin{equation}\label{eqn:volume}
\frac{\vol_{g_{t_0}} \big( B_{g_{t_0}}(z_0, R_0) \big)}{\vol_{g_{t_0}}(X)}\ge  C R_0^{\frac{2n p}{p-n}},
\end{equation}
for some constant $C>0$ depending only on $n, p, K, \omega_X, \chi$ and $\| e^F\|_{L^1(\log L)^p}$.
\end{theorem}
We remark that if $p\to \infty$, then the power exponent of $R_0$ in \eqref{eqn:volume} tends to $2n$, which is natural since the metrics $\omega_t$ are close to Euclidean ones when $R_0$ is small. However, an example on Riemann surfaces (i.e. $n=1$) shows that this power exponent $\frac{2n p}{p-n}$ is sharp in some sense (see {\bf Example 3.1} in section \ref{section 3}). If $e^F\in L^q(X,\omega_X^n)$ for some $q>1$, the right side of \eqref{eqn:volume} can be made as $C_\epsilon R_0^{2n + \epsilon}$ for any $\epsilon>0$ with the constant $C_\epsilon>0$ depending additionally on $\epsilon.$

\smallskip

The estimate (\ref{eqn:volume}) can be viewed as an analogue of Perelman's $\kappa$-noncollapsing theorem \cite{P} as the following. Let $g(t)$ be the smooth solution of the Ricci flow on a real $n$-dimensional compact Riemannian $M$ for $t\in [0, T)$ with the initial metric $g_0$. Let  $(x_0, t_0) \in M\times [0, T)$ and $r_0\in (0, (t_0)^{1\over 2})$. If the scalar curvature satisfies
$$\textnormal{R} < r_0^{-2}, ~\textnormal{on}~ B_{g(t_0)}(x_0, r_0),$$
then
$$  \vol_{g(t_0)}\left(B(x_0, t_0, r_0)\right) \geq \kappa r_0^n,$$
for some $\kappa>0$ that only depends on $n$ and $\nu[g_0, 2T]$, Perelman's $\mu$-functional at the initial time.

We define the {\em radius of Ricci curvature lower bound} at $z$ for the metric $g_t$ by 
$$\bar r_{g_t} (z) = \sup\{r>0 ~|~ \ric(g_t)\ge -\frac{1}{r^2} \text{ in } B_{g_t}(z, r)\}$$
as an analogue of the curvature radius introduced in \cite{CT}.  An equivalent way to state Theorem \ref{thm:main} is that under the same setup
$$\frac{\vol_{g_t} \big( B_{g_t}(z, \bar r_{g_t} (z) ) \big)}{\vol_{g_t}(X)}\ge  C \bar r_{g_t} (z) ^{\frac{2n p}{p-n}}, \text{ for any }z\in X,$$
for some constant $C>0$ depending only on $n, p, \omega_X, \chi$ and $\| e^F\|_{L^1(\log L)^p}$.

\smallskip

The main idea of the proof of Theorem \ref{thm:main} is motivated by that in \cite{GPT, GPTa}, that is, we compare the K\"ahler metric $\omega_t$ to some auxiliary complex Monge-Amp\`ere equation, with the function on the right-side being truncated squared distance function of $g_t$. Additionally we also need the Riemannian geometric tools like the Laplacian comparison and volume comparison theorems.

Theorem \ref{thm:main} immediately implies the following diameter bound established in \cite{FGS},  by choosing $R_0=1$ in Theorem \ref{thm:main}.

\begin{corollary}\label{cor:main}
For any $p>n$, if $\omega_t$ solves (\ref{eqn:MA}) for $t\in (0,1]$ and the Ricci curvature $\ric(g_t)$ is bounded below by 
$$\ric(g_t)\ge - K^2  $$
 for some $K \ge 0$,  then there exists $C>0$ depending only on $n, p, K, \omega_X, \chi$ and $\| e^F\|_{L^1(\log L)^p}$ such that 
\begin{equation}\label{eqn:volume1}
\textnormal{Diam}(X,g_t) \leq C
\end{equation} 
for all $t\in (0,1]$.
\end{corollary}

We also establish a local gradient estimate as an application of Theorem \ref{thm:main}.   We consider the following complex Monge-Amp\`ere equation
\begin{equation}\label{eqn:fixed MA}(\omega_X + \ddbar \varphi)^n = e^F \omega_X^n,\quad \sup_X\varphi= 0 \end{equation}
on a compact K\"ahler manifold $(X, \omega_X)$ of complex dimension $n$ with $F\in C^\infty(X)$ satisfying the normalization condition $\int_X e^F \omega_X^n = \int_X \omega^n$.

\begin{theorem}\label{thm:main2} 
 Let $\varphi$ be the solution of equation (\ref{eqn:fixed MA}) and $g$ be the K\"ahler metric corresponding to the K\"ahler form $\omega=\omega_X+ \ddbar \varphi$.  Given $p>n$ and $R_0\in (0,1]$, if 
 $$\ric(g)\ge - \frac{K^2}{R_0^2}, ~on ~B_g(z_0, 2R_0) $$
for some $K\geq 0$ and $z_0\in X$,   then 
$$|\nabla \varphi|_g^2 \le C,\text{ on }B_g(z_0, R_0)$$
for some constant $C>0$ depending on $n, p, \omega_X, \| e^F\|_{L^1(\log L)^p}$, $K$ and $R_0$.
\end{theorem}

We will prove Theorem \ref{thm:main} in section \ref{section 2}. In section \ref{section 3}, we construct an example and show that the exponent in (\ref{eqn:volume}) is sharp. In section \ref{section 4}, we will prove Theorem \ref{thm:main2} and discuss some other applications of  Theorem \ref{thm:main} and Corollary \ref{cor:main} on the diameter bound of K\"ahler metrics satisfying certain complex Monge-Amp\`ere equations.

\section{Proof of Theorem \ref{thm:main}}\label{section 2}
Since $\hat \omega_t$ may not be K\"ahler, we define the ``envelope'' associated to $\hat\omega_t$:
\begin{equation}\label{eqn:envelope}\mathcal V_t = \sup\{ v | ~ v\in PSH(X,\hat \omega_t),\, v\le 0  \}.\end{equation}
It is known that for each $t\in (0,1]$, $\mathcal V_t$ is a $C^{1,1}$ function. 
%
We recall the following uniform $L^\infty$-estimates on $\varphi_t$, the solution to \eqref{eqn:MA}.
\begin{lemma}[\cite{FGS, GPTW}]\label{lemma L infinity}
There is a constant $C_0>0$ depending only on $n, p, \omega_X,\chi$ and $\|e^F\|_{L^1(\log L)^p}$ such that 
$$\sup_X |\varphi_t - \mathcal V_t| \le C_0,\quad \forall t\in (0,1].$$

\end{lemma}
We are now ready to prove Theorem \ref{thm:main}.
\begin{proof} We break the proof into four steps.

\medskip

\noindent{\bf Step 1.} We fix a family of smooth positive functions $\eta_k: \mathbb R\to \mathbb R_+$ such that $\eta_k(x)$ converges uniformly and decreasingly to the function $x\cdot \chi_{\mathbb R_+}(x)$ as $k\to\infty$. We solve the auxiliary complex Monge-Amp\`ere equations
\begin{equation}\label{eqn:a MA}
(\hat \omega_t + \ddbar \psi_{t,k})^n = c_t \frac{ \eta_k\big( R_0^2 - d_t^2   \big)  }{A_{k, t}} e^F \omega_X^n,\quad \sup_X\psi_{t, k} = 0,
\end{equation}
where $0< A_{k,t} = \frac{c_t}{V_t} \int_X \eta_k\xk{R_0^2 - d_t^2} e^F \omega_X^n$ is the normalizing constant making the equation \eqref{eqn:a MA} solvable by Yau's theorem \cite{Y}. Here we write $d_t(x) = d_{g_t}(x, z_0)$ to be the geodesic distance of $x$ to the fixed point $z_0$ under the {\em varying} metric $g_t$. We note that although the positive function on the right-hand side of \eqref{eqn:a MA} is only Lipschitz but not necessarily smooth, the solution $\psi_{t,k}$ is still in $C^{2,\alpha}(X)$ for some $\alpha>0$, which follows from the regularity theory of complex Monge-Amp\`ere equations (see e.g. \cite{GPS, CH}). 
We observe that by dominated convergence theorem
$$A_{k,t} \to A_t: =  \frac{c_t}{V_t} \int_{ B_{g_t}(z_0, R_0)  } (R_0^2 - d_t^2) e^F \omega_X^n,\quad \text{as }k\to\infty.$$

\smallskip

\noindent{\bf Step 2.} We aim to compare $\psi_{t,k}$ with the solution $\varphi_t$. As in \cite{GPT, GPTa}, we look at the test function 
\begin{equation}\label{eqn:test}
\Phi : = -\varepsilon \xk{ -\psi_{t,k} + \varphi_t + C_1  }^{\frac{n}{n+1}} + (R_0^2 - d_t^2),
\end{equation}
where $C_1 = C_0+1$ and $C_0>0$ is the constant in Lemma \ref{lemma L infinity}, and $\varepsilon>0$ is chosen as \begin{equation}\label{eqn:eps}
\varepsilon = \bk{ \frac {n^3 + (n+1) (4n+2K)   }{n^2}    }^{\frac n{n+1}} A_{k,t}^{\frac{1}{1+n}} = : C_2 A_{k, t}^{\frac{1}{1+n}} ,
\end{equation}
where we fix the constant $C_2>0$ which depends on $n$ and $K$.  
As an initial observation we note that on $X$
$$ -\psi_{t,k} + \varphi_t + C_1 = - (\psi_{t, k} - \mathcal V_{t}) + (\varphi_t - \mathcal V_t) + C_0+1\ge  1,$$
by Lemma \ref{lemma L infinity} and the fact $\psi_{t,k}\le \mathcal V_t$ for each $k$. Therefore the function $\Phi<0$ on $X\backslash B_{g_t}(z_0, R_0)$.

\smallskip

We claim that $\Phi\le 0$ on $X$.  Let $x_{\max}\in X$ be a maximum point of $\Phi$. If $x_{\max}\not \in B_{g_t}(z_0, R_0)$, we are done. So we may assume $x_{\max}\in B_{g_t}(z_0, R_0)$. Applying Calabi's trick if necessary \cite{SY}, we may assume $d_t^2$ is smooth at $x_{\max}$. Then at $x_{\max}$ we have (write $\Delta = \Delta_{\omega_t}$)
\begin{align*}
0&\ge \Delta_{\omega_t} \Phi(x_{\max}) = \frac{n \varepsilon}{n+1} \xk{-\psi_{t,k} 
+ \varphi_t + C_1 }^{-\frac{1}{1+n}} (\Delta \psi_{t,k} - \Delta \varphi_t) \\
&\quad  + \frac{n\varepsilon}{(n+1)^2} \xk{-\psi_{t,k} + \varphi_t + C_1   }^{-\frac{n+2}{n+1}} |\nabla (\psi_{t,k} - \varphi_t)|^2_{\omega_t} - \Delta d_t^2\\
&\ge \frac{n \varepsilon}{n+1} \xk{-\psi_{t,k} 
+ \varphi_t + C_1 }^{-\frac{1}{1+n}} (\tr_{\omega_t} \omega_{\psi_{t,k}} - \tr_{\omega_t} \omega_t) -  2d_t \Delta d_t - 2\\
& \ge  \frac{n^2 \varepsilon}{n+1} \xk{-\psi_{t,k} 
+ \varphi_t + C_1 }^{-\frac{1}{1+n}} \bk{\frac{\omega^n_{\psi_{t,k}}}{\omega_t^n}}^{1/n} -  \frac{n^2 \varepsilon}{n+1} \xk{-\psi_{t,k} 
+ \varphi_t + C_1 }^{-\frac{1}{1+n}} \\
&\quad   -       2d_t \big ( \frac{2n-1}{d_t} + \frac{K}{R_0} \big) - 2\\
&\ge  \frac{n^2 \varepsilon}{n+1} \xk{-\psi_{t,k} 
+ \varphi_t + C_1 }^{-\frac{1}{1+n}} \bk{\frac{ \eta_k\big( R_0^2 - d_t^2   \big)  }{A_{k, t}} }^{1/n} -  \frac{n^2 \varepsilon}{n+1}   -   4n - 2K \\
&\ge  \frac{n^2 \varepsilon}{n+1} \xk{-\psi_{t,k} 
+ \varphi_t + C_1 }^{-\frac{1}{1+n}} \bk{\frac{ R_0^2 - d_t^2   }{A_{k, t}} }^{1/n} -  \frac{n^2 \varepsilon}{n+1}   -   4n - 2K 
\end{align*}
where we write $\omega_\phi = \hat \omega_t + \ddbar \phi$ for a function $\phi\in PSH(X,\hat \omega_t)$,   in the fourth line we apply the arithmetic-geometric inequality and the fact $\tr_{\omega_t} \omega_t = n$, in the fifth line we apply the Laplace comparison theorem \cite{SY} of the distance function which holds under our assumption that $\ric(\omega_t)\ge - K^2 R_0^{-2}$ on the geodesic ball $B_{g_t}(z_0, 2R_0)$, in the sixth line we use the equations satisfied by $\psi_{t,k}$ and $\varphi_t$ and in the last line we apply the choice of $\eta_k$ which satisfies $\eta_k(s)\ge s$ for $s\ge 0$. Hence at $x_{\max}$ we have
\begin{equation*}
R_0^2 - d_t^2 \le A_{k,t} \bk{n+ \frac{(n+1) (4n+2K)}{ n^2 \varepsilon }}^n (-\psi_{t,k} + \varphi_t + C_1)^{\frac n{n+1}} < \varepsilon  (-\psi_{t,k} + \varphi_t + C_1)^{\frac n{n+1}},
\end{equation*}
by the choice the $\varepsilon$ in \eqref{eqn:eps}. This finishes the proof of the claim that $\Phi\le 0$ on $X$. 

\smallskip

\noindent{\bf Step 3.} From $\Phi\le 0$ we infer that on $B_{g_t}(z_0, R_0)$
$$\frac{  (R_0^2 - d_t^2 )^{\frac{n+1}{n}}}{A_{k,t}^{1/n}} \le C_2 ^{\frac{n+1}{n}} ( -\psi_{t,k} + \varphi_t + C_1   ).   $$ 
We can view $\psi_{t,k}$ as a $C_3\omega_X$-PSH function for some $C_3>0$ depending only on $\chi$ and $\omega_X$. So the H\"ormander-Tian estimate \cite{Ti, H} holds for each $\psi_{t,k}$. We thus have a small constant $\alpha = \alpha(\omega_X, \chi, K)>0$ such that the following Trudinger-type inequality holds
\begin{equation}\label{eqn:1}
\int_{B_{g_t}(z_0, R_0)} e^{ \alpha \frac{  (R_0^2 - d_t^2 )^{\frac{n+1}{n}}}{A_{k,t}^{1/n}}   }\omega_X^n \le \int_X e^{ \alpha C_2^{\frac{n+1}{n}} ( -\psi_{t,k} + \varphi_t +C_1  ) } \omega_X^n\le C,
\end{equation}
for some uniform constant $C>0$, where we have chosen $\alpha>0$ small so that $\alpha C_2^{\frac{n+1}{n}}< \text{the }\alpha$-invariant of $(X, C_3\omega_X)$.

\smallskip

\noindent{\bf Step 4.} It then follows from a generalized Young's inequality that
$$v^p e^F \le e^F (1+ |F|^p) + C_p e^{2v}$$
which applied to $v= \alpha \frac{  (R_0^2 - d_t^2 )^{\frac{n+1}{n}}}{2 A_{k,t}^{1/n}}$ yields by \eqref{eqn:1} that
\begin{equation}\label{eqn:2}
\int_{B_{g_t}(z_0, R_0)} (R_0^2  - d_t^2) ^{\frac{n+1}{n} p} e^F \omega_X^n \le C A_{k,t}^{p/n},
\end{equation}
for some $C>0$ depending on $n, p, \omega_X, \chi, K$ and $\| e^F\|_{L^1(\log L)^p}$. Letting $k\to\infty$, \eqref{eqn:2} implies that
\begin{equation}\label{eqn:3}
\int_{B_{g_t}(z_0, R_0)} (R_0^2  - d_t^2) ^{\frac{n+1}{n} p} e^F \omega_X^n \le C A_{t}^{p/n}.
\end{equation}
On the other hand, by Holder inequality we have
\begin{align*}
A_t = \frac{c_t}{V_t}\int_{B_{g_t}(z_0, R_0)} (R_0^2  - d_t^2) e^F& \le \bk{  \int_{B_{g_t}(z_0, R_0)} (R_0^2  - d_t^2) ^{\frac{n+1}{n} p} e^F\omega_X^n }^{\frac{n}{p(n+1)}} \cdot \bk{ \int_{B_{g_t}(z_0, R_0)} e^F \omega_X^n  }^{1/q}\\
&\le C A_t^{1/(n+1)} \bk{ \int_{B_{g_t}(z_0, R_0)} e^F \omega_X^n  }^{1/q}
\end{align*}
where $q= \frac{p(n+1)}{p(n+1) - n}$ is the conjugate of $p$. We thus conclude that 
\begin{equation}\label{eqn:final}
A_t\le C \bk{ \int_{B_{g_t}(z_0, R_0)} e^F \omega_X^n  }^{\frac{n+1}{n q}}=C \bk{ \int_{B_{g_t}(z_0, R_0)} e^F \omega_X^n  }^{1 + \frac{p-n}{np}}.
\end{equation}
Note that $c_t/V_t$ is uniformly bounded, so \eqref{eqn:final} shows that there exists a uniform constant $C>0$ such that
\begin{equation*}
 \int_{B_{g_t}(z_0, R_0)} (R_0^2 - d_t^2)e^F \omega_X^n \le C \bk{ \int_{B_{g_t}(z_0, R_0)} e^F \omega_X^n  }^{1 + \frac{p-n}{np}}.
\end{equation*}
In particular, we obtain
\begin{equation}\label{eqn:final 1}
R_0^2 \int_{B_{g_t}(z_0, R_0/2)}e^F \omega_X^n \le C \bk{ \int_{B_{g_t}(z_0, R_0)} e^F \omega_X^n  }^{1 + \frac{p-n}{np}}.
\end{equation}
Multiplying $c_t$ on both sides of \eqref{eqn:final 1} we get
\begin{equation}\label{eqn:final 2}
R_0^2 \vol_{g_t}( B_{g_t}(z_0, R_0/2)  ) \le \frac{C}{c_t^{\frac{p-n}{np}}} \xk{\vol_{g_t} ( B_{g_t}(z_0, R_0)  )}^{1+\frac{p-n}{np}}.
\end{equation}
By volume comparison \cite{SY}, the function $(0, 2R_0) \ni r \mapsto r^{-2n} e^{- \frac{K}{R_0} r  } \vol_{g_t} (B_{g_t}(z_0, r)) $ is non-increasing, which implies
$$\vol_{g_t} ( B_{g_t}(z_0, R_0)  )\le C(n, K) \vol_{g_t} ( B_{g_t}(z_0, R_0/2)  ).$$ Combined with \eqref{eqn:final 2}, this implies that
\begin{equation}\label{eqn:final 3}
R_0^{\frac{2np}{p-n}} V_t\le C \vol_{g_t} \xk{B_{g_t}(z_0, R_0)  },
\end{equation}
where as usual $V_t = \int_X \omega_t^n$ and $C>0$ is a uniform constant. This finishes the proof of Theorem \ref{thm:main}.

\end{proof}

\section{An example}\label{section 3} 

In this section, we will construct an example and demonstrate that the exponent $\frac{2np}{p-n}$ is sharp in the estimate of Theorem \ref{thm:main}. In other words, the estimate in Theorem \ref{thm:main} may fail with the exponent replaced by $\frac{2np}{p-n} - \epsilon$ for any $\epsilon>0$.

\medskip

\noindent{\bf Example 3.1.}\,  Let $\mathbf D\subset \mathbb C\subset \mathbb{CP}^1$ be the disk with radius $1/2$. Consider the function $\varphi(z) = (-\log |z|^2)^{-a}$ for some $a>0$. We calculate the ``metric'' defined by $\ddbar \varphi$:
$$\omega = \ddbar \varphi = a(a+1) \frac{idz\wedge d\bar z}{|z|^2 (-\log |z|^2)^{a+2}} = e^F idz\wedge d\bar z.$$ 
Straightforward calculations show that $\| e^F\|_{L^1(\log L)^p(\mathbf D)}$ is bounded for any $1< p < a+1$ and is unbounded when $p\ge a+1$. Fix a point $z_0\in \mathbf D\backslash \{0\}$ close to $0$, and $w\in \mathbf D$ with $\arg z_0 = \arg w$. By the rotational symmetry of $\omega$, we see that
\begin{align*}d_\omega(z_0, w) =&  \sqrt{a(a+1)} \int_{\min\{|z_0|, |w|\}}^{\max\{|z_0|, |w|\}} \frac{dr }{ r (-\log r^2)^{1 + \frac a 2}  }\\
=& \sqrt{ \frac{a+1}{2^a a}  }\bk{ \frac{1}{(- \log \max \{|z_0|, |w|\})^{a/2} }  -  \frac{1}{(- \log \min \{|z_0|, |w|\})^{a/2} }    }.
\end{align*}
In particular, letting $w\to 0$ we see that 
$$d_\omega(z_0, 0) = \sqrt{ \frac{a+1}{2^a a}  }  \frac{1}{(- \log |z_0|)^{a/2} } =: 6 R_0>0. $$
Take $z_0^{\pm} \in \mathbf D$ with the same arguments as $z_0$ and $|z_0^{+}| = |z_0|^{2^ {2/a}} < |z_0|$ and $|z_0^{-}| = |z_0|^{(2/3)^{a/2}} > |z_0|$ so that $(-\log |z_0^+|)^{a/2} = 2 (- \log |z_0|)^{a/2}$ and $(-\log |z_0^-|)^{a/2} = \frac{2}{3} (- \log |z_0|)^{a/2}$. Then it follows that
$$d_\omega(z_0, z_0^\pm) = \frac 1 2 \sqrt{ \frac{a+1}{2^a a}  }  \frac{1}{(- \log |z_0|)^{a/2} }  = 3 R_0.$$
On the other hand, the length of the circles around $0$ in the annulus $\{|z_0^+| \le |z| \le |z_0^-|\}$ is given by (again by the rotational symmetry of $\omega$ these circles are $\omega$-geodesics)
$$L_{\omega}(\text{circle}) = \sqrt{a(a+1)}\int_0^{2\pi} \frac{d\theta}{(-\log |z|^2)^{1+ \frac a 2}}= \frac{2\pi \sqrt{a(a+1)} }{(-\log |z|^2)^{1+ \frac a 2}}< \frac{R_0}{10}$$
if $|z_0|>0$ small enough. This implies that
$$B_\omega(z_0, 2 R_0) \subset \{|z_0^+| \le |z| \le |z_0^-|\} \subset B_\omega(z_0, 4 R_0) \subset \mathbf D\backslash \{0\}.$$
By straightforward calculations, we have
\begin{align*}\ric(\omega) = & \ddbar \log \xk{ |z|^2 (-\log |z|^2)^{a+2}  } = - \frac{a+2}{|z|^2 (-\log |z|^2)^2} idz\wedge d\bar z\\
= & -\frac{a+2}{a(a+1)} (-\log |z|^2)^a \cdot \omega\\
\ge & - \frac{K^2}{R_0^2} \omega, \quad \text{in the annulus }  \{|z_0^+| \le |z| \le |z_0^-|\} = : \mathcal A,
\end{align*}
where $K>0$ is a constant depending only on $a>0$. Finally we calculate the $\omega$-volume of the above annulus as
$$\vol_\omega(\mathcal A) = \int_{|z_0^+| } ^{|z_0^-|} \frac{2\pi r dr }{r^2 (-\log r^2 )^{a+2}}= \frac{C_a}{(-\log |z_0|)^{a+1}} = C_a' R_0^{\frac{2(a+1)}{a}} . $$
Note that $e^F \in L^1(\log L)^p$ for {\bf any} $1<p< a+1$ and $\vol_\omega(\mathcal A)\ge \vol_\omega(B_\omega(z_0, 2 R_0))$. We see that the exponent $\frac{2np}{p-n}$ of $R_0$ in Theorem \ref{thm:main} is sharp. Moreover, the Ricci curvature $\ric(\omega)$ is not bounded below on the whole $\mathbf D$ since it decays to $-\infty$ near $0$. 

\medskip

Though $\omega$ in this example is a `singular' K\"ahler metric on a local domain, we can regularize it near $0$ to make it a genuine K\"ahler metric, and glue it to $\mathbb {CP}^1$ to get an example on compact K\"ahler manifolds (cf. e.g. \cite{Sze}). One can naturally generalize the above example to higher dimensions and indeed the exponent $\frac{2np }{p - n}$ is sharp.

\section{Applications of Theorem \ref{thm:main}}\label{section 4}
We discuss some geometric applications of the noncollapsing result in Theorem \ref{thm:main}. We will show a local gradient estimate of the K\"ahler potential, if the Ricci curvature is bounded below locally. Under certain assumption on the Ricci curvature lower bound, we will prove the
 diameter bound  and a  local noncollapsing result for K\"ahler metrics along the normalized K\"ahler-Ricci flow on minimal K\"ahler manifolds.

\subsection{MA equations with a fixed background K\"ahler metric} As the first application of Theorem \ref{thm:main}, we prove Theorem \ref{thm:main2} as generalization and a new proof of the global gradient estimate  in \cite{FGS}. 

Let $(X,\omega_X)$ be a given compact K\"ahler manifold. We consider the following complex Monge-Amp\`ere equation
\begin{equation}\label{eqn:fixed MA}(\omega_X + \ddbar \varphi)^n = e^F \omega_X^n,\quad \sup_X\varphi= 0.\end{equation}
Under the assumption of locally Ricci curvature lower bound, we prove the following local gradient estimate on $\varphi$. We write $\omega = \omega_X + \ddbar \varphi$, which satisfies \eqref{eqn:fixed MA} and denote $g$ the associated Riemannian metric of $\omega$. We recall the statement of Theorem \ref{thm:main2} below.

\begin{theorem}\label{prop 3.1}
Given $p>n$ and $R_0\in (0,1]$, suppose $\ric(g)\ge - K^2 / R_0^2$ on the geodesic ball $B_g(z_0, 2R_0)$, then 
$$|\nabla \varphi|_g^2 \le C,\text{ on }B_g(z_0, R_0)$$
for some constant $C>0$ depending on $n, p, \omega_X, \| e^F\|_{L^1(\log L)^p}$, $K$ and $R_0$.
\end{theorem}
We first recall the $L^\infty$ estimate on $\varphi$ in  \cite{K, GPT}
\begin{equation}\label{eqn:L}
\| \varphi\|_{L^\infty} \le \overline{C}(n, p, \omega_X, \| e^F\|_{L^1(\log L)^p} ).
\end{equation}
To prove Theorem \ref{prop 3.1}, we need the following mean value inequality \cite{SY}. 
\begin{lemma}\label{lemma 3.1}
Let $(M,g)$ be a Riemannian manifold such that $B_g(p, 2R_0)$ is relatively compact in $M$ with $R_0\in (0,1]$. Suppose $u\ge 0$ is a nonnegative function on $B_g(p, 2 R_0)$ satisfying $\Delta_g u\ge - A^2/R_0^2$ and $\ric(g)\ge - K^2/R_0^2$ on $B_g(p, 2R_0)$, then the following mean value inequality holds
$$\sup_{B_g(p, \tau R_0)} u^2 \le \frac{C}{\vol_g(B_g(p, 2R_0))} \int_{B_g(p, 2 R_0)} (u^2+1) dV_g,$$ for some $C>0$ depending on $\tau\in [1,4/3]$, $K$ and $A$.
\end{lemma}
\begin{proof}
We consider $v = u+1\ge 1$, and $v$ satisfies $\Delta_g v \ge - (A^2 /R_0^2)v$. We look at the product manifold $B_g(p, 2R_0)\times \mathbb R$ with the metric $\hat g = g + ds ^2$, where $s\in \mathbb R$ is the natural coordinate. Define $\hat v = e^{A s/R_0} v$ to be function on $B_g(p, 2R_0)\times \mathbb R$, and it satisfies $\Delta_{\hat g} \hat v \ge 0$. Clearly $\ric(\hat g) \ge - K^2/R_0^2$ and we can then apply the standard mean value inequality (cf. Theorem 6.2, Ch. 2, \cite{SY}) to conclude that
\begin{align*}\sup_{B_g(p, \tau R_0)\times (-\tau R_0, \tau R_0)} \hat v^2  \le & \frac{C}{R_0 \vol_g(B_g(p, 2 R_0))} \int_{-2R_0}^{2R_0} \int_{B_g(p, 2R_0)} \hat v^2 dV_g ds\\
\le & \frac{C}{ \vol_g(B_g(p, 2 R_0))} \int_{B_g(p, 2R_0)} v^2 dV_g ds \end{align*}
from which the lemma follows.

\end{proof}
\begin{lemma}\label{lemma 4.2}
Under the same assumptions as in Theorem \ref{prop 3.1}, we have
$$\tr_\omega \omega_X \le \exp\big( { C R_0^{- np/(p-n)}   }\big),\text{ on }B_g(z_0, 3R_0/2),$$ for some $C>0$ depending  on $n, p, \omega_X, \| e^F\|_{L^1(\log L)^p}$ and $K\ge 0$.
\end{lemma}
\begin{proof}
It follows from the Schwarz-lemma type inequality that on $B_{g}(z_0, 2R_0)$
\begin{equation}\label{eqn:use}
\Delta_\omega \tr_\omega \omega_X \ge -\frac{K^2}{R_0^2} \tr_\omega \omega_X - C_0(\tr_\omega \omega_X)^2 +\frac{ |\nabla \tr_\omega \omega_X|^2_\omega}{ \tr_\omega \omega_X},
\end{equation}
where $C_0>0$ is an upper bound of the bisectional curvature of $\omega_X$ on $B_g(z_0, 2R_0)$. Straightforward calculations show that
$$\Delta_\omega \log \tr_\omega \omega_X \ge -\frac{K^2}{R_0^2} - C_0 \tr_\omega \omega_X.$$ Hence the function $u: = (\log \tr_\omega \omega_X)_+ -  C_0 \varphi $ satisfies
$$\Delta u \ge - \frac{K^2}{R_0^2} - C_0n = -\frac{A^2}{R_0^2}, \quad \text{with } A^2 = K^2 + C_0 n R_0^2. $$
We apply Lemma \ref{lemma 3.1} to conclude that (we denote $B_R = B_g(z_0, R)$ for simplicity)
\begin{align*}
\sup_{B_{3R_0/2}} [(\log \tr_\omega \omega_X)_+]^2 \le & 2 \sup_{B_{3R_0/2}} u^2 + 2 (C_0 \overline C)^2\\
\le & \frac{C}{\vol_g(B_{2R_0})} \int_{B_{2R_0}} [(\log \tr_\omega \omega_X)_+]^2 dV_g + C\\
\le & \frac{C}{R_0^{2np/(p-m)}} \int_X \tr_\omega \omega_X \; \omega^n + C \le \frac{C}{R_0^{2np/(p-n)}} ,
\end{align*}
where in the third line we use the calculus inequality $[(\log x)_+]^2\le x$ for any $x>0$ and Theorem \ref{thm:main}. Hence we have
$$\sup_{B_{3R_0/2}} \tr_\omega \omega_X \le \exp\big( { C R_0^{-np/(p-n)}   }\big).$$
\end{proof}
\begin{proof}[Proof of Theorem \ref{prop 3.1}]
By Bochner formula, we have on $B_g(z_0, 2 R_0)$
\begin{align*}\Delta _\omega |\nabla \varphi|^2_\omega = & |\nabla \nabla \varphi|_\omega^2 + |\nabla \bar \nabla \varphi |_\omega^2 + \ric_g(\nabla \varphi, \bar \nabla \varphi) - 2 Re\innpro{ \nabla\varphi, \bar \nabla \tr_\omega \omega_X   }_\omega\\
\ge &  |\nabla \nabla \varphi|_\omega^2 + |\nabla \bar \nabla \varphi |_\omega^2  - K^2 R_0^{-2 } |\nabla \varphi|^2 _\omega - 2 |\nabla \varphi| |\nabla \tr_\omega \omega_X|_\omega,
\end{align*}
which by the Kato's inequality\footnote{The Kato's inequality states that $2|\nabla |\nabla \varphi||_\omega^2 \le |\nabla\nabla \varphi|_\omega^2 + |\nabla \bar \nabla \varphi|_\omega^2$, which can be verified by the Cauchy-Schwarz inequality. It suffices to verify this inequality when $|\nabla \varphi| >0$. In normal coordinates of $\omega$, we calculate 
\begin{align*}
4|\nabla \varphi|^2 |\nabla |\nabla \varphi||^2 & = |\nabla |\nabla \varphi|^2|^2 \\
& = (\varphi_i \varphi_{\bar i})_{\bar j} (\varphi_k \varphi_{\bar k})_j\\
& = \varphi_{i\bar j} \varphi_{\bar i}\varphi_{kj}\varphi_{\bar k} + \varphi_{i\bar j} \varphi_{\bar i} \varphi_k \varphi_{\bar k j } + \varphi_{i}\varphi_{\bar i \bar j} \varphi_{kj} \varphi_{\bar k} + \varphi_{i}\varphi_{\bar i \bar j} \varphi_k \varphi_{\bar k j}   \\
&\le 2 |\nabla \varphi|^2 |\nabla \nabla \varphi| |\nabla \bar \nabla\varphi | + |\nabla \varphi|^2 |\nabla \bar \nabla \varphi|^2 + |\nabla \varphi|^2 |\nabla \nabla \varphi|^2
\\
& = |\nabla \varphi|^2 (|\nabla \nabla \varphi| + |\nabla \bar \nabla \varphi|)^2\le 2|\nabla \varphi|^2 (|\nabla \nabla \varphi|^2 + |\nabla \bar \nabla \varphi|^2).
\end{align*} Cancelling $2|\nabla \varphi|^2$ on both sides gives the desired Kato's inequality. } implies that
$$\Delta_\omega |\nabla \varphi|_\omega \ge - K^2 R_0^{-2} |\nabla \varphi|_\omega - |\nabla \tr_\omega \omega_X|_\omega.$$
Define $u = |\nabla \varphi|_\omega + \varphi^2 + (\tr_\omega \omega_X)^2$. We calculate using \eqref{eqn:use} 
\begin{align*}
\Delta_\omega u\ge &  - K^2 R_0^{-2} |\nabla \varphi|_\omega - |\nabla \tr_\omega \omega_X| + 2 |\nabla \varphi|_\omega^2 + 2n \varphi - 2\varphi \tr_\omega \omega_X \\
& + 2 |\nabla \tr_\omega \omega_X|^2_\omega  - 2 \frac{K^2}{R_0^2} (\tr_\omega \omega_X)^2 - 2 C_0(\tr_\omega \omega_X)^3 \\
\ge & - C(R_0),\quad \text{in } B_g(z_0, 3R_0/2),
\end{align*}
where in the last inequality the Schwarz inequality, \eqref{eqn:L} and Lemma \ref{lemma 4.2}. It then follows from Lemma \ref{lemma 3.1} that
$$\sup_{B_{R_0}} u^2 \le \frac{C(R_0)}{ \vol_g(B_g(z_0, \frac{3R_0}{2})) } \int_{B_{3R_0/2}} (1+u^2) \omega^n\le \frac{C(R_0)}{R_0^{2np/(p-n)} } \int_X |\nabla \varphi|_\omega^2 \omega^n + C(R_0).$$
The proof is complete by observing that 
$$\int_X |\nabla \varphi|_\omega^2 \omega^n = \int_X  (-\varphi) (\omega - \omega_X)\wedge \omega^{n-1}\le \int_X(-\varphi)\omega^n$$
and the latter integral is bounded due to \eqref{eqn:L}.
\end{proof}

We remark that if the Ricci curvature $\ric(\omega)$ is bounded below on $X$, then Proposition \ref{prop 3.1} implies a global gradient estimate of $\varphi$, which was proved in \cite{FGS} using maximum principle. Note that in \cite{FGS}, the assumption $\ddbar F \le A \omega_X$ was made to ensure the Ricci curvature being bounded below. With this extra assumption, the bound on $\| e^F\|_{L^1(\log L)^p}$ with $p>n$ together with Corollary \ref{cor:main} imply the diameter bound of $(X,\omega)$, where $\omega = \omega_\varphi$ satisfies \eqref{eqn:fixed MA}.

\smallskip

Finally we mention that when the function $e^F$ on the right-side of \eqref{eqn:fixed MA} belongs to $L^q$ for some $q>1$, the diameter bound of $\omega_\varphi$ has been proved in \cite{Li}, using the H\"older continuity of the solution $\varphi$ established in \cite{K1, DDK} (see \cite{GPTW1}). When $\| e^F\|_{L^1(\log L)^p}\le A$ for some $A>0$ and $p>3n$, the diameter bound of $\omega_\varphi$ has been obtained in \cite{GPTW1} after they establish the modulus of continuity of the K\"ahler potentials. 


\subsection{Normalized K\"ahler-Ricci flow} As in the previous subsection, we assume $X$ is a minimal K\"ahler manifold (i.e. $K_X$ is {\em nef}, and we do not assume $K_X$ is semi-ample) with $\omega_X$ a given K\"ahler metric. Let $\Omega$ be a smooth volume form such that $\chi = \ddbar \log \Omega\in |K_X|$ is a representative of the canonical class of $X$. Multiplying a constant to $\omega_X$ if necessary we may assume that $\chi \le \omega_X$. We consider the following normalized K\"ahler-Ricci flow
\begin{equation}\label{eqn:KRF}
\frac{\partial \omega_t}{\partial t} = -\ric(\omega_t) - \omega_t,\quad \omega_t|_{t=0} = \omega_X.
\end{equation}
It is well-known \cite{TZ} that \eqref{eqn:KRF} exists for all $t\in [0,\infty)$ and \eqref{eqn:KRF} is equivalent to the following parabolic complex Monge-Amp\`ere equation
\begin{equation}\label{eqn:pMA}
\frac{\partial \varphi_t}{\partial t} = \log \frac{ (\chi + e^{-t} (\omega_X - \chi) + \ddbar \varphi_t)^n   }{e^{- (n-\nu)t}\Omega} - \varphi_t, \quad \varphi_t|_{t=0} = 0,
\end{equation}
with $\omega_t = \chi + e^{-t} (\omega_X - \chi) + \ddbar \varphi_t >0$.

\begin{lemma}\label{lemma 3.3}
There is a uniform constant $C>0$ such that  
$$\sup_X\dot \varphi_t =\sup_X \frac{\partial \varphi_t}{\partial t} \le C \quad \text{and }\quad \sup_X \varphi_t\le C.$$
\end{lemma}
\begin{proof}
Taking $\frac{\partial }{\partial t}$ on both sides of \eqref{eqn:pMA}, we get
$$\frac{\partial \dot \varphi_t}{\partial t} = \Delta_{\omega_t} \dot \varphi_t - e^{-t} \tr_{\omega_t} (\omega_X - \chi) - \dot\varphi_t + (n-\nu),\quad \dot\varphi_t|_{t=0} = 0,$$ so by maximum principle and $\omega_X - \chi\ge 0$ it follows that for $\varphi_{t,\max} = \max_X \dot\varphi_t$
$$\frac{d}{dt} \dot\varphi_{t,\max}\le - \dot\varphi_{t,\max} + (n-\nu)$$
multiplying both sides by $e^t$ and integrating over $t$ it yields that $\dot\varphi_{t,\max}\le (n-\nu) (1-e^{-t})\le n-\nu$, which gives the upper bound of $\dot\varphi_t$. To see the upper bound of $\varphi_t$, we calculate (denote $V =\int_X \Omega$)
\begin{align*}
\frac{d}{dt}\frac{1}{V}\int_{X} \varphi_t\Omega & =\frac 1 V \int_X \log \frac{ \omega_t^n   }{e^{- (n-\nu)t}\Omega} \Omega - \frac 1 V \int_X \varphi_t\Omega\\
& = (n-\nu) t + \frac 1 V \int_X \log \frac{ \omega_t^n   }{\Omega} \Omega - \frac 1 V \int_X \varphi_t\Omega\\
&\le (n-\nu) t + \log \bk{\frac 1 V \int_X \omega_t^n} - \frac 1 V \int_X \varphi_t \Omega\\
&\le (n-\nu) t + \log \bk{\frac C V e^{- (n-\nu) t}} - \frac 1 V \int_X \varphi_t \Omega \\
&\le C - \frac 1 V \int_X \varphi_t \Omega,
\end{align*}
where in the third line above we applied Jensen's inequality. Integrating both sides of the inequality above it follows that $\frac 1 V \int_X \varphi_t \Omega\le C$ for some uniform $C>0$. The desired upper bound of $\varphi_t$ follows from mean value theorem. 
\end{proof}
We re-write the equation \eqref{eqn:pMA} as
\begin{equation}\label{eqn:MA2}
\big(\chi + e^{-t}(\omega_X - \chi) + \ddbar \varphi_t\big)^n = e^{-t (n-\nu)}e^{\varphi_t + \dot\varphi_t} \Omega.
\end{equation}
Integrating both sides of  \eqref{eqn:MA2} and applying Lemma \ref{lemma 3.1} we get $\int_X e^{\varphi_t} \Omega\ge c>0$ for some uniform $c>0$, which implies that $\sup_X \varphi_t \ge -C$ where $C>0$ is independent of $t$. We can now apply the $L^\infty$-estimate of $\varphi_t$ as the solution to the ``elliptic'' complex MA equation \eqref{eqn:MA2} with $e^F: = e^{\varphi_t + \dot\varphi_t}\le C$ from Lemma \ref{lemma 3.3}, and it follows that
\begin{equation}\label{eqn:L infty}
\sup_X |(\varphi_t - \sup_X \varphi_t) - \mathcal V_t| \le C \quad \Rightarrow \quad \sup_X |\varphi_t - \mathcal V_t|\le C,
\end{equation}
in which we use $|\sup_X \varphi_t|\le C$ and $\mathcal V_t$ is the envelope associated to $\chi + e^{-t}(\omega_X - \chi)$. With \eqref{eqn:L infty} we are ready to state the local noncollapsing along the normalized K\"ahler-Ricci flow \eqref{eqn:KRF}.
\begin{prop}\label{prop:KRF}
Let $X$ be a minimal K\"ahler manifold and $\omega_t$ be the solution to the K\"ahler-Ricci flow \eqref{eqn:KRF}. Let the assumptions be as above. Then the following holds. For any $\epsilon>0$, $t\in [0,\infty)$, $z_0\in X$, and $R_0\in (0,1]$, if $\ric(\omega_t)\ge - {K^2}/{R_0^2}$ on $B_{\omega_t}(z_0, 2R_0)$, then there exists a constant $c_\epsilon = c_\epsilon(n, \chi, \omega_X, K, \Omega, \epsilon)>0$ such that 
$$\frac{\vol_{\omega_t} ( B_{\omega_t}(z_0, R_0)  )} {V_t} \ge c_\epsilon R_0^{2n + \epsilon}.$$
\end{prop}

The proof of Proposition \ref{prop:KRF} is almost the same as that of Theorem \ref{thm:main}, given the $L^\infty$-estimate \eqref{eqn:L infty} of $\varphi_t$. In fact, the proof is even simpler since the function $e^F = e^{\varphi_t + \dot\varphi_t}$ is bounded in $L^\infty$-norm. We leave the details to interested readers.

If the Ricci curvature $\omega_t$ is uniformly bounded below for $t\in [0,\infty)$, the local noncollapsing in Proposition \ref{prop:KRF} will imply the uniform diameter bound of $(X,\omega_t)$, thus providing a new proof of one of the results in \cite{G}, which studied the K\"ahler-Ricci flow on minimal manifolds with numerical dimension $\nu = n$. By volume comparison, it will follow that $\vol_{\omega_t}(B_{\omega_t} (z_0, r)) \ge c V_t r^{2n}  $ for any $r\in (0,1]$. This volume decay property appears in the assumptions of the main theorems in \cite{Ha}. It suggests that this extra assumption may be superfluous for K\"ahler-Ricci flow.

\bigskip

\noindent{\bf Acknowledgement:} The authors would like to thank Professor S. Ko\l{}odziej for suggesting them generalize the diameter bound in \cite{FGS} to more general right-side of the MA equation, which motivated them to study this problem. We also thank V. Datar and  J. Sturm for their interest in this work. We are grateful to the anonymous referee for his or her valuable suggestions and comments.

\end{document}